\documentclass[a4paper,12pt,reqno]{amsart}

\usepackage[T1]{fontenc}
\usepackage[utf8]{inputenc}
\usepackage{lmodern}
\usepackage[english]{babel}

\usepackage[%
	left=2.5cm,       
	right=2.5cm,      
	top=3.5cm,        
	bottom=3.5cm,     
	heightrounded,    
	bindingoffset=0mm 
]{geometry}

\usepackage{amsmath}
\usepackage{amssymb}
\usepackage{mathtools}
\usepackage{mathrsfs}
\usepackage[normalem]{ulem}

\usepackage{microtype}

\usepackage[
colorlinks=true,
urlcolor=teal,
citecolor=magenta,
linkcolor=teal,
breaklinks=true]
{hyperref}

\usepackage[noabbrev,capitalize,nameinlink]{cleveref}

\usepackage{bookmark}

\usepackage[initials,msc-links,
]{amsrefs}

\usepackage[dvipsnames]{xcolor}
\usepackage{graphicx}
\usepackage{url}

\usepackage{enumitem}

\numberwithin{equation}{section}

\theoremstyle{plain}
\newtheorem{theorem}{Theorem}[section]
\newtheorem{lemma}[theorem]{Lemma}

\newtheorem{corollary}[theorem]{Corollary}

\theoremstyle{definition}
\newtheorem{definition}[theorem]{Definition}

\newtheorem{remark}[theorem]{Remark}

\newcommand{\di}{\,\mathrm{d}}

\newcommand{\e}{\varepsilon}

\newcommand{\N}{\mathbb{N}}

\newcommand{\R}{\mathbb{R}}
\newcommand{\sob}[1]{\mathcal{S}^{#1}}

\newcommand{\weakstarto}{\overset{\star}{\rightharpoonup}}

\DeclarePairedDelimiter{\set}{\{}{\}}

\newcommand{\closure}[2][3]{%
  {}\mkern#1mu\overline{\mkern-#1mu#2}}

\usepackage[
final
]{showlabels}

\showlabels{bib}

\allowdisplaybreaks
\begin{document}

\title[On a weighted version of the BBM formula]{On a weighted version of the
BBM formula}

\author[G.~Stefani]{Giorgio Stefani}
\address[G.~Stefani]{Università degli Studi di Padova, Dipartimento di
Matematica ``Tullio Levi-Civita'', Via Trieste 63, 35121 Padova (PD), Italy}
\email{giorgio.stefani@unipd.it {\normalfont or}
giorgio.stefani.math@gmail.com}

\date{\today}

\dedicatory{Dedicated to Silvia, on the occasion of our marriage.}



\thanks{\textit{Acknowledgements}.
The author thanks L.~Gennaioli for carefully reading and providing valuable
suggestions on a preliminary version of the manuscript---and for a delicious
\textit{tiramisù} in Treviso.
The author is member of the Istituto Nazionale di Alta Matematica (INdAM),
Gruppo Nazionale per
l’Analisi Matematica, la Probabilità e le loro Applicazioni (GNAMPA), and has
received funding from INdAM under the INdAM--GNAMPA Project 2025 \textit{Metodi
variazionali per problemi dipendenti da operatori frazionari isotropi e
anisotropi} (grant agreement No.\ CUP\_E53\-240\-019\-500\-01), and from the
European Union -- NextGenerationEU and the University of Padua under the 2023
STARS@UNIPD  Starting Grant Project \textit{New Directions in Fractional
Calculus -- NewFrac} (grant agreement No.\ CUP\_C95\-F21\-009\-990\-001).
}

\begin{abstract}
We prove a weighted version of the Bourgain--Brezis--Mironescu (BBM) formula, both in
the pointwise and $\Gamma$-convergence sense, together with a compactness
criterion for energy-bounded sequences.
The non-negative weights need only be $L^\infty$ convergent to a bounded and
uniformly continuous limit.
We apply the BBM formula to show a Poincaré-type inequality and the stability of the first eigenvalues relative to the energies.
Finally, we discuss a non-local analogue of the weighted BBM formula.
\end{abstract}

\maketitle

\section{Introduction}

This is a spin-off of the recent work~\cite{GS25} by L.\ Gennaioli and the
author concerning \textit{sharp} conditions for the validity of the
Bourgain--Brezis--Mironescu (BBM, for short) formula~\cites{BBM01,D02,P04a}.
Our main aim is to show that the \textit{sufficiency} part of the BBM formula
proved in~\cite{GS25} remains true even if the energies under consideration are
modified in order to include some suitable \emph{weights}.
As an application of our results, we prove a Poincaré-type inequality and the stability of the first eigenvalues relative to the energies.
Finally, we also discuss the validity of a non-local analogue of the weighted BBM formula.

\subsection{Weighted functionals}

Throughout the paper, given $p\in[1,\infty)$, we consider a family of kernels 
$(\rho_k)_{k\in\N}\subset L^1_{\rm loc}(\R^N)$ such that $\rho_k\ge0$ for every
$k\in\N$ and, unless otherwise stated, we assume that  
\begin{equation}
\label{eq:ddp_p}
\sup_{R>0}
\limsup_{k\to\infty}
R^p\int_{\R^N}\frac{\rho_{k}(z)}{R^p+|z|^p}\di z<\infty
\end{equation}
and, for some $\alpha\ge0$,
\begin{equation}
\label{eq:ddp_alpha}
\nu_k=\rho_{k}\mathscr L^N\weakstarto\alpha\delta_0
\quad
\text{in 
$\mathscr M_{\rm loc}(\R^N)$ as $k\to\infty$}.
\end{equation}
Moreover, we fix a family of non-negative \emph{weights} $(w_k)_{k\in\N}\subset
L^\infty(\R^{2N})$ such that $w_k\to w$ in $L^\infty(\R^{2N})$ as $k\to\infty$
for some non-negative \emph{limit weight} $w\in C_b(\R^{2N})$.
We assume that
\begin{equation}
\label{eq:omega_w}
|w(x,y)-w(x',y')|
\le 
\omega(|x-x'|+|y-y'|),
\quad
\text{for all $x,x',y,y'\in\R^N$},
\end{equation}
for some non-decreasing modulus of continuity
$\omega\colon[0,\infty)\to[0,\infty)$ such that 
\begin{equation}
\label{eq:modcont}
\displaystyle\lim_{t\to0^+}\omega(t)=\omega(0)=0.
\end{equation}

Observe that we do \textit{not} impose any symmetry assumption on the weights
$(w_k)_{k\in\N}$ and~$w$, as one may be interested in possibly
\textit{non-symmetric} interactions~\cite{DDG24}.  

With the above notation in force, we define the \emph{weighted} functionals
$\mathscr F_{k,p}^{w_k}\colon L^p(\R^N)\to[0,\infty]$ by letting
\begin{equation}
\label{eq:F_w}
\begin{split}
\mathscr F_{k,p}^{w_k}(u)
&=
\int_{\R^N}
\int_{\R^N}
\frac{|u(x)-u(y)|^p}{|x-y|^p}
\,
\rho_k(x-y)\,w_k(x,y)\di x\di y
\\
&=
\int_{\R^N}\frac{\|u(\cdot+z)-u\|_{L^p(w_k^z)}^p}{|z|^p}\,\rho_k(z)\di z,
\end{split}
\end{equation}
for all $u\in L^p(\R^N)$ and $k\in\N$, where, for each $z\in\R^N$,
\begin{equation}
\label{eq:notation_w}
\|u\|_{L^p(w_k^z)}^p
=
\int_{\R^d}|u(x)|^p\,w_k^z(x,x)\di x
=
\int_{\R^d}|u(x)|^p\,w_{k}(x,x+z)\di x.
\end{equation}
The functionals $(\mathscr F^w_{k,p})_{k\in\N}$, as well as the norm
$\|\,\cdot\,\|_{L^p(w^z)}$, are analogously defined. 
The \emph{unweighted} functionals in~\cite{GS25}, where $w_k=w=1$ for all
$k\in\N$, are given by $(\mathscr F^1_{k,p})_{k\in\N}$.

\subsection{Local framework}
As proved in~\cite{GS25} (also refer to~\cite{DDP24} for the case $p=2$ and
to~\cite{F25} for \textit{radially symmetric} kernels), the unweighted
functionals $(\mathscr F^1_{k,p})_{k\in\N}$ converge to
\begin{equation}
\label{eq:dir}
\mathscr D_{p,w}^\mu(u)
=
\int_{\mathbb S^{N-1}}\|\sigma\cdot Du\|_{L^p}^p\di\mu(\sigma)
\end{equation}
in the pointwise and $\Gamma$-sense on $\sob{p}(\R^N)$ with respect to the
$L^p$ topology for some Radon measure $\mu\in\mathscr M(\mathbb S^{N-1})$ defined on the $(N-1)$-dimensional unit sphere
$\mathbb S^{N-1}$ if and only if the non-negative kernels $(\rho_k)_{k\in\N}$ statisfy~\eqref{eq:ddp_p} and~\eqref{eq:ddp_alpha}
(possibly, up to a subsequence). 
For the notion of $\Gamma$-convergence, refer to~\cite{GS25}*{Sec.~2.5} for an
account and to~\cites{Braides02,DalMaso93} for a complete treatment. 
Above and in the rest of the work, as in~\cite{GS25}, we set 
\begin{equation*}
{\sob{p}}(\R^N)
=
\begin{cases}
W^{1,p}(\R^N)
&
\text{for}\ p>1,
\\[1ex]
BV(\R^N)
&
\text{for}\ p=1,
\end{cases}
\end{equation*}
and, for each $u\in\sob p(\R^N)$ and $\sigma\in\mathbb S^{N-1}$, 
\begin{equation}
\label{eq:D_sigma}
\|\sigma\cdot Du\|_{L^p}^p
=
\begin{cases}
\displaystyle
\int_{\R^N}|\sigma\cdot Du|^p\di x
&
\text{for}\
p>1,
\\[3ex]
\displaystyle|\sigma\cdot Du|(\R^N)
&
\text{for}\
p=1.
\end{cases}
\end{equation}
As customary, in~\eqref{eq:D_sigma} above $Du$ denotes the distributional
gradient of $u\in {\sob{p}}(\R^N)$.
In particular, if $p=1$, then $Du$ may be a finite (vectorial) Radon measure on
$\R^N$.

We underline that the measure $\mu\in\mathscr M(\mathbb S^{N-1})$
appearing in~\eqref{eq:dir} is uniquely determined by the family $(\rho_k)_{k\in\N}$
(possibly, up to a subsequence) and can be (formally) defined as
\begin{equation}
\label{eq:mu}
\mu(E)
=
\lim_{\delta\to0^+}
\lim_{k\to\infty}
\int_E\left(\int_0^\delta\rho_k(\sigma r)\,r^{N-1}\di r\right)\di\mathscr
H^{N-1}(\sigma)
\end{equation}
for every set $E\subset\mathbb S^{N-1}$ measurable with respect to the
$(N-1)$-dimensional Hausdorff measure~$\mathscr H^{N-1}$.
We refer to~\cite{GS25}*{Lem.~2.9} for the precise construction of~$\mu$.
Because of~\eqref{eq:mu}, the nature of the limit measure $\mu$
in~\eqref{eq:dir} does \textit{not} depend on the chosen weights
$(w_k)_{k\in\N}$ and $w$.
For this reason, from now on we tacitly assume that the measure $\mu\in\mathscr
M(\mathbb S^{N-1})$ is uniquely identified by~\eqref{eq:ddp_p}
and~\eqref{eq:ddp_alpha}.

Our first main result, \cref{res:bbm_w} below, proves that~\eqref{eq:ddp_p} and~\eqref{eq:ddp_alpha} are still
\textit{sufficient} for the  pointwise and $\Gamma$-convergence of the
functionals $(\mathscr F^{w_k}_{k,p})_{k\in\N}$ in~\eqref{eq:F_w} to a weighted
version of the energy in~\eqref{eq:dir}; namely,\begin{equation}
\label{eq:dir_w}
\mathscr D_{p,w}^\mu(u)
=
\int_{\mathbb S^{N-1}}\|\sigma\cdot Du\|_{L^p(w^0)}^p\di\mu(\sigma),
\end{equation}
where $\mu\in\mathscr M(\mathbb S^{N-1})$ is as in~\eqref{eq:mu} and, for every
$\sigma\in\mathbb S^{N-1}$, as in~\eqref{eq:D_sigma},
\begin{equation}
\label{eq:D_w}
\|\sigma\cdot Du\|_{L^p(w^0)}^p
=
\begin{cases}
\displaystyle
\int_{\R^N}w^0\,|\sigma\cdot Du|^p\di x
&
\text{for}\
p>1,
\\[3ex]
\displaystyle\int_{\R^N}w^0\di|\sigma\cdot Du|
&
\text{for}\
p=1.
\end{cases}
\end{equation}
In~\eqref{eq:D_w} above, the function $w^0\in C_b(\R^N)$ is defined as
$w^0(x)=w(x,x+0)=w(x,x)$ for every $x\in\R^N$ according to the notation
introduced in~\eqref{eq:notation_w}.
Moreover, here and in the following, given any $R>0$, we define the ($L^p$
closed) subspaces
\begin{equation*}
L^p_R(\R^N)
=
\set*{u\in L^p(\R^N) : \operatorname{supp}u\subset\closure{B}_R},
\quad
\sob{p}_R(\R^N)
=
\sob{p}(\R^N)\cap L^p_R(\R^N).
\end{equation*}

\begin{theorem}[Weighted BBM formula]\label{res:bbm_w}
Let $p\in[1,\infty)$ and $(\rho_k)_{k\in\N}$ be as above.
The following hold:
\begin{enumerate}[label=(\roman*),itemsep=1ex,topsep=1ex]

\item
\label{item:bbm_w_limsup} 
if $u\in\sob{p}(\R^N)$, then
$\displaystyle\limsup_{k\to\infty}
\mathscr F_{k,p}^{w_k}(u)
\le
\mathscr D_{p,w}^\mu(u)$;

\item
\label{item:bbm_w_G-liminf}
if $(u_k)_{k\in\N}\subset L^p(\R^N)$ is such that $u_k\to u$ in $L^p(\R^N)$ as
$k\to\infty$ for some $u\in {\sob{p}}(\R^N)$ and
$\displaystyle\limsup_{k\to\infty}
\mathscr F_{k,p}^1(u_k)<\infty$ then
$\displaystyle
\liminf_{k\to\infty}
\mathscr F_{k,p}^{w_k}(u_k)
\ge 
\mathscr D_{p,w}^\mu(u)$;

\item
\label{item:bbm_w_G-liminf_R}
if $w>0$ on $\R^{2N}$, $R>0$ and $(u_k)_{k\in\N}\subset L^p_R(\R^N)$ is such
that $u_k\to u$ in $L^p(\R^N)$ as $k\to\infty$ for some $u\in\sob{p}(\R^N)$,
then $\displaystyle\limsup_{k\to\infty}
\mathscr F_{k,p}^1(u_k)<\infty$ and therefore 
$\displaystyle
\liminf_{k\to\infty}
\mathscr F_{k,p}^{w_k}(u_k)
\ge 
\mathscr D_{p,w}^\mu(u)$.
\end{enumerate}
As a consequence, as $k\to\infty$, the functionals 
$(\mathscr F_{k,p}^{w_k})_{k\in\N}$ converge to $\mathscr D_{p,w}^{\mu}$
pointwise on $\sob{p}(\R^N)$ and in the $\Gamma$-sense on $\sob{p}_R(\R^N)$ for
every $R>0$.
\end{theorem}

We underline that the functionals~\eqref{eq:F_w} are not invariant under
translations, due to the presence of the weights $(w_k)_{k\in\N}$ and $w$. 
However, despite this lack of regularity, assumptions~\eqref{eq:ddp_p}
and~\eqref{eq:ddp_alpha} are still \textit{sufficient} for the validity of a
BBM formula in this case.
Besides, we observe that the assumptions~\eqref{eq:ddp_p}
and~\eqref{eq:ddp_alpha} cannot be weakened as, in fact, by the main result
of~\cite{GS25}, both ones are also \textit{necessary} for the validity of
\cref{res:bbm_w} in the unweighted case $w_k=w=1$ for all $k\in\N$.

Finally, we remark that the validity of the statements~\ref{item:bbm_w_limsup},
\ref{item:bbm_w_G-liminf} and~\ref{item:bbm_w_G-liminf_R} in \cref{res:bbm_w}
for \textit{some} measure $\mu\in\mathscr M(\mathbb S^{N-1})$ cannot imply the
conditions~\eqref{eq:ddp_p} and~\eqref{eq:ddp_alpha}, unless one imposes some
additional (strong) assumptions on the weights, at least by requiring that
$\inf_{\R^{2N}}w>0$.
Indeed, at the level of generality we are presently working, one may choose
$w_k=w$ for all $k\in\N$ and $w(x,y)=\omega(|x-y|)$ for all $x,y\in\R^N$, where
$\omega\colon[0,\infty)\to[0,\infty)$ is a bounded non-decreasing function
satisfying~\eqref{eq:modcont}.
In this case, we may rewrite 
\begin{equation*}
\mathscr F^{w_k}_{k,p}(u)
=
\int_{\R^N}
\frac{\|u(\cdot+z)-u\|_{L^p}^p}{|z|^p}\,\widetilde{\rho_k}(z)\di z
=
\widetilde{\mathscr F_{k,p}}(u)
\end{equation*}
where $\widetilde{\rho_k}=\rho_k\,\omega(|\cdot|)$ for all $k\in\N$.
Note that the functionals $(\widetilde{\mathscr F_{k,p}})_{k\in\N}$ above are
precisely the ones considered in~\cite{GS25}, so the validity of the BBM
formula (in our case, the validity of \cref{res:bbm_w}) is equivalent to the
conditions~\eqref{eq:ddp_p} and~\eqref{eq:ddp_alpha} for the family
$(\widetilde{\rho_k})_{k\in\N}$. 
In other words, by assuming~\eqref{eq:ddp_p} and~\eqref{eq:ddp_alpha}, we are
implicitly deciding the roles of the players in the game: $(\rho_k)_{k\in\N}$
play as the kernels and $(w_k)_{k\in\N}$ play as the weights.

\subsection{Compactness}

Our second main result, \cref{res:comp_w} below, is a simple compactness
criterion for sequences of functions with bounded energy.

\begin{definition}
\label{def:p-compact}
Let $p\in[1,\infty)$. 
A sequence 
$(u_k)_{k\in\N}\subset L^p(\R^N)$ is \emph{$p$-energy bounded} if 
\begin{equation*}
\sup_{k\in\N}
\big(
\|u_k\|_{L^p}
+
\mathscr F_{k,p}^1(u_k)
\big)<\infty.
\end{equation*}
Moreover,
the family $(\rho_k)_{k\in\N}$ is \emph{$p$-compact} if any $p$-energy bounded
sequence in $L^p(\R^N)$ is compact in $L^p(E)$ for every compact set
$E\subset\R^N$.
\end{definition}

\begin{theorem}[Compactness]
\label{res:comp_w}
Let $(\rho_k)_{k\in\N}$ be $p$-compact, $w>0$ on $\R^{2N}$ and $R>0$.
If $(u_k)_{k\in\N}\subset L^p_R(\R^N)$ is such that
\begin{equation*}
\sup_{k\in\N}
\big(
\|u_k\|_{L^p}
+
\mathscr F_{k,p}^{w_k}(u_k)
\big)<\infty,
\end{equation*}
then $(u_k)_{k\in\N}$ is compact in $L^p(\R^N)$.
\end{theorem}

Examples of $p$-compact families of kernels $(\rho_k)_{k\in\N}$ can be found for instance in~\cites{BBM01,GS25,P04b,AABPT23}.
Here we only refer to~\cite{BBM01}*{Th.~5} and to~\cite{P04b}*{Ths.~1.2
and~1.3} for the case of \textit{radially symmetric} kernels, and
to~\cite{GS25}*{Th.~1.3} for possibly non-radially symmetric ones.
We also refer to~\cites{BS24,F25} for the purely non-local framework.

In the references above, the family $(\rho_k)_{k\in\N}$ is not only $p$-compact
(for some or for all~$p$), but also yields that any $L^p_{\rm loc}(\R^N)$ limit
of a $p$-energy bounded sequence is a more regular function; namely,
it may belongs to $\sob{p}(\R^N)$. 
This motivates the following terminology.

\begin{definition}
\label{def:strong_p-comp}
Let $p\in[1,\infty)$. 
The family $(\rho_k)_{k\in\N}$ is \emph{strongly $p$-compact} if it is
$p$-compact and such that any $L^p_{\rm loc}(\R^N)$ limit of a $p$-energy
bounded sequence is in $\sob{p}(\R^N)$.\end{definition}

As well-known (e.g., see~\cites{P04a,P04b}), the family of standard
\emph{fractional} kernels, defined as  
\begin{equation}
\label{eq:frac_kernels}
\rho_s(z)=\frac{1-s}{|z|^{N-(1-s)p}},
\quad
\text{for $z\in\R^N$ and $s\in(0,1)$},
\end{equation}
is an example of strongly $p$-compact family for every $p\in[1,\infty)$.

By combining \cref{res:bbm_w,res:comp_w} with \cref{def:strong_p-comp}, we
obtain the following result.

\begin{corollary}
\label{res:strong_bbm_w}
Let $p\in[1,\infty)$, $w>0$ on $\R^{2N}$ and $R>0$.
Assume that $(\rho_k)_{k\in\N}$ is a strongly $p$-compact family.
Then, the following hold:
\begin{enumerate}[label=(\roman*),itemsep=1ex,topsep=1ex]

\item
\label{item:strong_bbm_comp}
if $(u_k)_{k\in\N}\subset L^p_R(\R^N)$ is such that
\begin{equation}
\label{eq:strong_bbm_comp}
\sup_{k\in\N}
\big(
\|u_k\|_{L^p}
+
\mathscr F_{k,p}^{w_k}(u_k)
\big)<\infty,
\end{equation}
then $(u_k)_{k\in\N}$ is compact in $L^p(\R^N)$ and any of its $L^p(\R^N)$
limits is in $\sob{p}(\R^N)$;

\item
\label{item:strong_bbm_liminf}
if $(u_k)_{k\in\N}\subset L^p_R(\R^N)$ is such that $u_k\to u$ in $L^p(\R^N)$
as $k\to\infty$ for some $u\in L^p(\R^N)$, then  
$\displaystyle
\liminf_{k\to\infty}
\mathscr F_{k,p}^{w_k}(u_k)
\ge 
\mathscr D_{p,w}^\mu(u)$;

\item 
\label{item:strong_bbm_limsup}
if $u\in\sob{p}(\R^N)$, then
$\displaystyle\limsup_{k\to\infty}
\mathscr F_{k,p}^{w_k}(u)
\le
\mathscr D_{p,w}^\mu(u)$.

\end{enumerate} 
\end{corollary} 

A particular case of \cref{res:strong_bbm_w} has been recently obtained by
A.~Kubin, G.~Saracco and the author in~\cite{KSS25}, with a completely
different proof, for the fractional kernels~\eqref{eq:frac_kernels} and weights
defined as $w_k(x,y)=f_k(x)\,f_k(y)$ for all $k\in\N$ and $w(x,y)=f(x)\,f(y)$,
for every $x,y\in\R^N$, where $(f_k)_{k\in\N}\subset
L^\infty(\R^N;[0,\infty))$, $f\in\mathrm{Lip}_b(\R^N;(0,\infty))$  and $f_k\to
f$ in $L^\infty(\R^N)$ as $k\to\infty$ (for the unweighted case, refer
to~\cite{CDKNP23}*{Th.~2.1} for $p=2$ and to~\cite{BPS16}*{Th.~3.1} for
$p\in(1,\infty)$).
For strictly related results, see~\cites{CCLP23,DL21} for the fractional
\emph{Gaussian} case, and~\cite{K24} for weights depending on negative powers
of the distance from the boundary of the domain of integration.

\begin{remark}[A generalization of~\cite{KSS25}*{Th.~1.2}]
\cref{res:strong_bbm_w} can be exploited to prove the stability of the gradient flows relative to the Hilbertian energies $(\mathscr F^{w_k}_{k,2})_{k\in\N}$ and $\mathscr D^\mu_{2,w}$, thus generalizing~\cite{KSS25}*{Th.~1.2} to the present setting.
Indeed, it is enough to follow the strategy outlined in~\cite{KSS25}*{Sec.~4} line by line up to the natural (minor) adaptations.
\end{remark}

\subsection{Poincaré inequality}

Motivated by~\cite{P04b}, we exploit \cref{res:strong_bbm_w}  to estalish a Poincaré-type
inequality for the energies $(\mathscr F^{w_k}_{k,p})_{k\in\N}$, see 
\cref{res:poincare_w} below.

In order to state our result, we need to introduce some notation.
From now on, we fix a (non-empty) bounded open set $\Omega\subset\R^N$ and, for
each $p\in[1,\infty)$, we define the subspaces
\begin{equation}
\label{eq:zero_spaces}
L^p_0(\Omega)
=
\set*{u\in L^p(\R^N) : u=0\ \text{a.e.\ on}\ \R^N\setminus\Omega},
\quad
\sob p_0(\Omega)
=
\set*{u\in\sob p(\R^N) : u\in L^p_0(\Omega)}.
\end{equation}
As well known, if $\Omega$ has Lipschitz regular boundary, then $\sob
p_0(\Omega)$ coincides with the closure of $C^\infty_c(\Omega)$ functions with
respect to the energy $u\mapsto\|Du\|_{L^p}$, but we do not need such
equivalence in the following.

\begin{definition}
\label{def:poincare}
Let $p\in[1,\infty)$.
The measure $\mu\in\mathscr M(\mathbb S^{N-1})$ is \emph{$(p,\Omega)$-Poincaré}
if there exists $C>0$ such that, for every $u\in\sob p_0(\Omega)$, it holds 
\begin{equation}
\label{eq:poincare_mu}
\|u\|_{L^p}^p\le C\,\mathscr D_p^\mu(u).
\end{equation}
In this case, we let $A_{p,\mu,\Omega}>0$ be the optimal constant for
which~\eqref{eq:poincare_mu} holds. 
\end{definition}

If $\mu\in\mathscr M(\mathbb S^{N-1})$ is $(p,\Omega)$-Poincaré for some
$p\in[1,\infty)$, then there exists $C>0$, depending on $A_{p,\mu,\Omega}$ and
$\inf_\Omega w^0>0$, such that 
\begin{equation}
\label{eq:poincare_mu_w}
\|u\|_{L^p}^p
\le 
C\,
\mathscr D^{\mu}_{p,w}(u)
\end{equation}
for all $u\in\sob p_0(\Omega)$.
We thus let $A_{p,\mu,\Omega,w}>0$ be the optimal constant for
which~\eqref{eq:poincare_mu_w} holds.

\begin{theorem}[Poincaré inequality]
\label{res:poincare_w}
Let $p\in[1,\infty)$.
Assume that $(\rho_k)_{k\in\N}$ is a strongly $p$-compact family and
$\mu\in\mathscr M(\mathbb S^{N-1})$ is $(p,\Omega)$-Poincaré.
Then, for every $\e>0$, then exists $k_\e\in\N$ such that 
\begin{equation*}
\|u\|_{L^p}^p
\le 
\left(A_{p,\mu,\Omega,w}+\e\right)
\,
\mathscr F_{k,p}^{w_k}(u)
\end{equation*}
for all $u\in L^p_0(\Omega)$ and  $k\ge k_\e$. 
\end{theorem}

It is worth observing that $\mu\in\mathscr M(\mathbb S^{N-1})$ is
$(p,\Omega)$-Poincaré as in \cref{def:poincare} for every $p\in[1,\infty)$ if
$\operatorname{span}(\operatorname{supp}\mu)=\R^N$, due
to~\cite{GS25}*{Lem.~2.1}.
This latter property, in turn, is ensured for example by assuming that  the
family $(\rho_k)_{k\in\N}$ has \emph{maximal rank} as
in~\cite{GS25}*{Def.~2.8}, see~\cite{GS25}*{Lem.~2.9} for the proof.
As remarked in~\cite{GS25}*{Sec.~1.2}, any radially symmetric family has
maximal rank, but non-radially symmetric families with maximal rank are known
(examples can be found in~\cite{P04a}).
In particular, as a consequence, \cref{res:poincare_w} applies to the radially symmetric family
of fractional kernels~\eqref{eq:frac_kernels}.

\subsection{Spectral stability}

Inspired by~\cite{BPS16}, we make use of \cref{res:strong_bbm_w} to also show a stability result for the first eigenvalues relative to the energies $(\mathscr F^{w_k}_{k,p})_{k\in\N}$ and~$\mathscr D^\mu_{p,w}$ and of the corresponding eigenfunctions, see \cref{res:eigen_conv} below.

Letting $p\in[1,\infty)$ and $\Omega\subset\R^N$ be as above, and keeping in mind the subspaces defined in~\eqref{eq:zero_spaces}, for each $k\in\N$ we let
\begin{equation}
\label{eq:eigeinv_k}
\lambda_{k,p}^{w_k}(\Omega)
=
\inf\set*{
\mathscr F^{w_k}_{k,p}(u) 
: u\in L^p_0(\Omega)\ \text{such that}\ \|u\|_{L^p}=1}\in[0,\infty)
\end{equation}
be the \emph{first eigenvalue} relative to the energy $\mathscr F^{w_k}_{k,p}$ on~$\Omega$, and we call any $u\in L^p_0(\Omega)$ attaining the infimum in~\eqref{eq:eigeinv_k} an \textit{eigenfunction} relative to $\lambda_{k,p}^{w_k}(\Omega)$.
Analogously, we let
\begin{equation}
\label{eq:eiegenv_w}
\lambda_{p,w}^\mu(\Omega)
=
\inf\set*{
\mathscr D_{p,w}^\mu(u) 
: u\in\sob p_0(\Omega)\ \text{such that}\ \|u\|_{L^p}=1}\in[0,\infty)
\end{equation}
be the \emph{first eigenvalue} relative to $\mathscr D^\mu_{p,w}$ on~$\Omega$, and we call any $u\in\sob p_0(\Omega)$ attaining the infimum in~\eqref{eq:eiegenv_w} an \textit{eigenfunction} relative to $\lambda^\mu_{p,w}(\Omega)$.

It is worth observing that $\lambda^{w_k}_{k,p}(\Omega)<\infty$ for all $k\in\N$ because, owing to~\eqref{eq:ddp_p}, we have that $\mathscr F^{w_k}_{k,p}(u)<\infty$ for any $u\in C^\infty_c(\Omega)$. 
Analogously, we also have that $\lambda^\mu_{p,w}(\Omega)<\infty$.
In addition, if $\mu\in\mathscr M(\mathbb S^{N-1})$ is $(p,\Omega)$-Poincaré as in \cref{def:poincare}, then $\lambda^\mu_{p,w}(\Omega)>0$ by~\eqref{eq:poincare_mu_w} and also $\lambda^{w_k}_{k,p}(\Omega)>0$ for all $k\in\N$ sufficiently large by \cref{res:poincare_w}.

\begin{theorem}[Spectral stability]
\label{res:eigen_conv}
Let $p\in[1,\infty)$ and assume that $(\rho_k)_{k\in\N}$ is a strongly $p$-compact family.
Then, it holds
\begin{equation}
\label{eq:eigenv_conv}
\lim_{k\to\infty}
\lambda_{k,p}^{w_k}(\Omega)
=
\lambda_{p,w}^\mu(\Omega).
\end{equation}
Moreover, if $u_k\in L^p_0(\Omega)$ is an eigenfunction relative to $\lambda_{k,p}^{w_k}(\Omega)$ 
for each $k\in\N$, then there exist $(u_{k_j})_{j\in\N}$ and an eigenfunction $u\in\sob p_0(\Omega)$ relative to $\lambda_{p,w}^\mu(\Omega)$ such that 
\begin{equation}
\label{eq:eigenf_conv}
\text{$u_{k_j}\to u$ in $L^p(\R^N)$ as $j\to\infty$}.
\end{equation}
\end{theorem}

In the unweighted setting $w_k=w=1$ for all $k\in\N$, and assuming that $(\rho_k)_{k\in\N}$ is the family of fractional kernels~\eqref{eq:frac_kernels}, 
\cref{res:eigen_conv} corresponds to~\cite{BPS16}*{Th.~1.2} for $m=1$, the substantial (and natural) difference between the two results being that the convergence in~\eqref{eq:eigenf_conv} holds in a stronger sense in~\cite{BPS16} (besides, observe that $p>1$ in~\cite{BPS16}).

\begin{remark}[A generalization of~\cite{BPS16}*{Th.~1.2}]
\cref{res:strong_bbm_w} allows to recover~\cite{BPS16}*{Th.~1.2} also for $m\ge2$ for the unweighted energies~$(\mathscr F^1_{k,p})_{k\in\N}$ and $\mathscr D^\mu_p$ for $p\in(1,\infty)$.
Indeed, up to assuming that $\rho_k/|\cdot|^p\notin L^1(\R^N)$ for all sufficiently large $k\in\N$, one can exploit the compactness criterion proved in~\cite{BS24}*{Th.~2.11} and follow  the proof of~\cite{BPS16}*{Th.~1.2} (based on the general results achieved in~\cite{DM14}) line by line up to the natural (minor) adaptations.
\end{remark}

\subsection{Non-local framework}

We end up with a generalization of the non-local BBM
formula obtained  in~\cite{GS25}*{Ths.~1.4 and~1.5}, see \cref{res:J} below.
Here and in the following, we let $A\subset\R^N$ be an open set and 
$J_k,J\colon A^2\to[0,\infty]$, with $k\in\N$ and
$A^2=A\times A\subset\R^{2N}$, be measurable functions.
We thus let 
\begin{equation*}
\mathscr J_{k,p}(u)
=
\int_A\int_A|u(x)-u(y)|^p\,J_k(x,y)\di x \di y
\end{equation*}
for all $k\in\N$ and, similarly,
\begin{equation*}
\mathscr J_{p}(u)
=
\int_A\int_A|u(x)-u(y)|^p\,J(x,y)\di x \di y,
\end{equation*}
for every $u\in L^p(A)$, and we define
\begin{equation*}
\mathcal W^{J,p}(A)=\set*{u\in
L^p(A):\mathscr J_p(u)<\infty}.
\end{equation*}

\begin{theorem}[Weighted non-local BBM formula]
\label{res:J}
Let $p\in[1,\infty)$ and let $J_k,J\colon A^2\to[0,\infty]$, $k\in\N$, be as above.
The following hold.
\begin{enumerate}[label=(\roman*),itemsep=1ex,topsep=1ex,leftmargin=6ex]

\item
\label{item:J_comp}
Assume that, for every $\e>0$, there exists $\delta>0$ such that
\begin{equation}
\label{eq:J_comp}
\inf_{k\in\N}J_k(x,y)
\ge 
\frac{1}{\e\delta^N}
\quad
\text{for $\mathscr L^{2N}$-a.e.\ $(x,y)\in A^2$\ s.t.\ $|x-y|<\delta$}.
\end{equation} 
If $(u_k)_{k\in\N}\subset L^p(A)$ is such that 
\begin{equation}
\label{eq:J_sup_comp}
\sup_{k\in\N}
\left(\|u_k\|_{L^p(A)}+
\mathscr J_{k,p}(u_k)\right)
<\infty,
\end{equation}
then $(u_k)_{k\in\N}$ is compact in $L^p(E)$ for every compact set $E\subset A$.

\item
\label{item:J_G-liminf}
Assume that 
\begin{equation}
\label{eq:J_liminf}
\liminf_{k\to\infty}J_k\ge J
\quad
\text{$\mathscr L^{2N}$-a.e.\ in $A^2$}.
\end{equation}
If $(u_k)_{k\in\N}\subset L^p(A)$ is such that $u_k\to u$ in $L^p(A)$
as $k\to\infty$ for some $u\in L^p(A)$, then
$\displaystyle\liminf_{k\to\infty}
\mathscr J_{k,p}(u_k)\ge\mathscr J_{p}(u)$.

\item
\label{item:J_lim} 
Assume that
\begin{equation}
\label{eq:J_lim}
J=\lim_{k\to\infty}J_k
\quad
\text{$\mathscr L^{2N}$-a.e.\ on $A^2$}
\end{equation}
and that either there exist $C>0$ such that  
\begin{equation}
\label{eq:J_dom}
\sup_{k\in\N}J_k\le CJ
\quad
\text{$\mathscr L^{2N}$-a.e.\ on $A^2$},
\end{equation}
or that there exist $H_+,H_-\subset A^2$, with $\mathscr L^{2N}(H_+\cap
H_-)=0$ and $\mathscr L^{2N}(A^2\setminus(H_+\cup H_-))=0$, such that 
\begin{equation}
\label{eq:J_mon}
J_k\le J_{k+1}
\
\text{a.e.\ on $H_+$}
\quad
\text{and}
\quad
J_k\ge J_{k+1}\
\text{a.e.\ on $H_-$}
\end{equation} 
for all $k\in\N$.
If $u\in L^p(A)$ is such that $\mathscr J_p(u)<\infty$, then 
$\displaystyle\lim_{k\to\infty}
\mathscr J_{k,p}(u)
=
\mathscr J_{p}(u)$. 

\end{enumerate}
Therefore, under~\eqref{eq:J_lim} and either~\eqref{eq:J_dom} or~\eqref{eq:J_mon}, as $k\to\infty$,  the functionals $(\mathscr
J_{k,p})_{k\in\N}$ converge pointwise and in the $\Gamma$-sense to $\mathscr J$
with respect to the $L^p$ topology in $\mathcal W^{J,p}(A)$.  
\end{theorem}

\cref{res:J} may be already known to experts, but we were not able to trace it
in the literature.
Some particular instances of \cref{res:J} can be found in~\cite{AV16}*{Th.~1.1(iii)}
and~\cite{CS23}*{Th.~5.9}.
Moreover, note that the statements~\ref{item:J_G-liminf} and~\ref{item:J_lim} of  \cref{res:J} can be \textit{verbatim} stated and proved in any measure space
$(A,\mathfrak m)$, while the statement~\ref{item:J_comp} requires some extra caution.
In fact, we underline that~\eqref{eq:J_comp} is only prototypical and  can be relaxed or variated in many ways, although we do not insist on this technical point.

We exploit \cref{res:J} to obtain a weighted counterpart
of~\cite{GS25}*{Ths.~1.4 and~1.5}, see \cref{res:nonloc_bbm_w} below (in which we do \textit{not} assume~\eqref{eq:ddp_p} and~\eqref{eq:ddp_alpha}). 
Similarly as in~\cite{GS25}, given a measurable function $\kappa\colon\R^N\to[0,\infty)$, we
let
\begin{equation*}
W^{\kappa,p}_w(\R^N)  
=  
\set*{u\in L^p(\R^N) : [u]_{W^{\kappa,p}_w}<\infty},  
\end{equation*}  
where\begin{equation*}  
[u]_{W^{\kappa,p}_w}  
=  
\left(\int_{\R^N}\int_{\R^N}|u(x)-u(y)|^p\,\kappa(x-y)\,w(x,y)\di x \di
y\right)^{1/p}.  
\end{equation*}

\begin{corollary}
\label{res:nonloc_bbm_w}
Let $(\rho_k)_{k\in\N}\subset L^1_{\rm loc}(\R^N)$ be such that $\rho_k\ge0$ for all $k\in\N$.
Assume that, for every $\e>0$, there exist $\delta>0$ such that
\begin{equation}
\label{eq:nonloc_bbm_w_comp}
\inf_{k\in\N}
\frac{\rho_k}{|\cdot|^p}
\ge
\frac{1}{\e\delta^N}
\quad
\text{$\mathscr L^N$-a.e.\ on $B_\delta$}, 
\end{equation}
and that, moreover, for some $C>0$, the family $(\rho_k)_{k\in\N}$ satisfies
\begin{equation}
\label{eq:nonloc_bbm_w_ass}
\sup_{k\in\N}\frac{\rho_k}{|\cdot|^p}
\le 
C\kappa
\quad
\text{and}
\quad
\lim_{k\to\infty}\frac{\rho_k}{|\cdot|^p}=\kappa
\quad
\text{a.e.\ on $\R^N$}.
\end{equation}
Then, the following hold:
\begin{enumerate}[label=(\roman*),itemsep=1ex,topsep=1ex]

\item
\label{item:nonloc_comp}
if $(u_k)_{k\in\N}\subset L^p(\R^N)$ is such that 
\begin{equation*}
\sup_{k\in\N}
\left(
\|u_k\|_{L^p}+\mathscr F^{w_k}_{k,p}(u_k)
\right)
<\infty,
\end{equation*}
then $(u_k)_{k\in\N}$ is compact in $L^p(B_R)$ for every $R>0$ and any of its $L^p_{\rm loc}(\R^N)$ limit belongs to $W^{\kappa,p}_{w}(\R^N)$;

\item
\label{item:nonloc_G-liminf}
if $(u_k)_{k\in\N}\subset L^p(\R^N)$ is such that $u_k\to u$ in $L^p(\R^N)$ as
$k\to\infty$ for some $u\in L^p(\R^N)$, then
$\displaystyle\liminf_{k\to\infty}
\mathscr F_{k,p}^{w_k}(u_k)\ge[u]_{W^{\kappa,p}_w}^p$;

\item
\label{item:nonloc_lim}
if 
$u\in W^{\kappa,p}_w(\R^N)$, then 
$\displaystyle\lim_{k\to\infty}
\mathscr F_{k,p}^{w_k}(u)
=
[u]_{W^{\kappa,p}_w}^p$.

\end{enumerate}
As a consequence, as $k\to\infty$, the functionals $(\mathscr
F^{w_k}_{k,p})_{k\in\N}$ converge pointwise and in the $\Gamma$-sense to
$[\,\cdot\,]_{W^{\kappa,p}_w}^p$ with respect to the $L^p$ topology in
$W^{\kappa,p}(\R^N)$. 
\end{corollary}

\section{Proofs of the results}

The rest of the paper is dedicated to the proofs of the results.
The main notation and the conventions adopted here are identical to the ones introduced in~\cite{GS25}*{Sec.~2}.

\subsection{Proof of \texorpdfstring{\cref{res:bbm_w}}{Theorem 1.1}}

\label{subsec:proof_bbm}

We begin by observing that it is enough to prove the
statements~\ref{item:bbm_w_limsup} and~\ref{item:bbm_w_G-liminf} of
\cref{res:bbm_w} only in the case $w_k=w$ for all $k\in\N$.
Indeed, this follows by combining~\cite{GS25}*{Th.~1.1} with the simple
inequality
\begin{equation}
\label{eq:w_diff_F}
\left|
\mathscr F_{k,p}^{w_k}(u)
-
\mathscr F_{k,p}^{w}(u)
\right|
\le
\mathscr F_{k,p}^1(u)\,\|w_k-w\|_{L^\infty}
\end{equation}
valid for all $k\in\N$ and $u\in L^p(\R^N)$.
Moreover, since we can rescale all the functionals by a positive constant
factor, it is not restrictive to further assume that $\|w\|_{L^\infty}=1$.

We are going to take advantage of the following  generalization
of~\cite{GS25}*{Lem.~2.2}, whose proof is briefly detailed below for the ease
of the reader.

\begin{lemma}
\label{res:ftc}
If $u\in\sob{p}(\R^N)$ and $z\in\R^N$, then the following hold:
\begin{enumerate}[label=(\roman*),itemsep=1ex,topsep=1.75ex]
\item 
\label{item:ftc1}
$\|u(\cdot+z)-u\|_{L^p(w^z)}^p
\le 
\|z\cdot Du\|_{L^p(w^0)}^p
+
\omega(|z|)\,
|z|^p\,\|Du\|_{L^p}^p$;
\item
\label{item:ftc2}
$\left|\|u(\cdot+z)-u\|_{L^p(w^z)}-\|z\cdot Du\|_{L^p(w^0)}\right|
\le 
\dfrac{|z|^2}{2}\,\|D^2u\|_{L^p}+|z|\,\omega(|z|)^{1/p}\,\|Du\|_{L^p}$.

\end{enumerate}
\end{lemma}

In the proof of \cref{res:ftc} we exploit the following result.
Here and below, we let $(\eta_j)_{j\in\N}\subset C^\infty_c(\R^N)$,
$\eta_j=j^N\eta(j\,\cdot\,)$ for all $j\in\N$, be a family of non-negative
mollifiers, with $\eta\in C^\infty_c(\R^N)$ such that $\eta\ge0$,
$\operatorname{supp}\eta\subset B_1$ and $\int_{\R^N}\eta\di x=1$.

\begin{lemma}
\label{res:bruco}
If $z\in\R^N$, $u\in\sob{p}(\R^N)$ and $u^j=u*\eta_j$ for all $j\in\N$, then
\begin{equation}
\label{eq:bruco}
\lim_{j\to\infty}
\|z\cdot Du^j\|_{L^p(w^0)}^p
= 
\|z\cdot Du\|_{L^p(w^0)}^p.
\end{equation} 
\end{lemma} 

\begin{proof}
Let us set $D_z u=z\cdot Du$ for brevity.
For $p=1$, the validity of~\eqref{eq:bruco} follows from the fact that
$|\eta_j*D_z u|\mathscr L^N\weakstarto|D_z u|$ in $\mathscr M(\R^N)$ as
$j\to\infty$ and  $w^0\in C_b(\R^N)$.
For $p>1$, instead, on the one hand, owing to Jensen's inequality and the
Dominated Convergence Theorem, we easily infer that 
\begin{equation*}
\limsup_{j\to\infty}
\|D_z u^j\|_{L^p(w^0)}^p
\le
\lim_{j\to\infty}
\int_{\R^N}(\eta^j*w^0)\,|D_z u|^p\di x
=
\|D_z u\|_{L^p(w^0)}^p.
\end{equation*}
On the other hand, since $(\eta_j*D_z u)(x)\to D_z u(x)$ as $j\to\infty$ for
a.e.\ $x\in\R^N$, owing to Fatou's Lemma, we also get that
\begin{equation*}
\liminf_{j\to\infty}
\int_{\R^N}|\eta_j*D_z u|^p\,w^0\di x
\ge 
\int_{\R^N}|D_z u|^p\,w^0\di x,
\end{equation*}
concluding the proof of ~\eqref{eq:bruco} in the case $p>1$.
\end{proof}

\begin{proof}[Proof of \cref{res:ftc}]
By \cref{res:bruco} and an approximation argument, we can assume $u\in
C^\infty(\R^N)\cap\sob{p}(\R^N)$ without loss of generality.
We prove the two statements separately. 

\vspace{1ex}
\textit{Proof of \ref{item:ftc1}}. 
We can estimate
\begin{equation*}
\begin{split}
\int_{\R^N}|u(x+z)&-u(x)|^p\,w^z(x)\di x
\le 
\int_0^1
\int_{\R^N}|Du(x+tz)\cdot z|^p\,w(x,x+z)\di x\di t
\\
&\le 
\int_0^1\int_{\R^N}|Du(x+tz)\cdot z|^p\,w(x+tz,x+tz)\di x\di t
\\
&\quad
+
\int_0^1\int_{\R^N}|Du(x+tz)\cdot z|^p\,|w(x,x+z)-w(x+tz,x+tz)|\di x\di t
\\
&\le
\|z\cdot Du\|_{L^p(w^0)}^p
+\omega(|z|)\,|z|^p\|Du\|_{L^p}^p
\end{split}
\end{equation*}
by the Fundamental Theorem of Calculus, Jensen's inequality
and~\eqref{eq:omega_w}.

\vspace{1ex}
\textit{Proof of \ref{item:ftc2}}.
Letting $\phi(t)=u(x+tz)$ for every $t\in[0,1]$, we can write 
\begin{equation*}
\phi(1)
=
\phi(0)
+ 
\phi'(0)
+
\int_0^1\phi''(t)\,(1-t)\di t.
\end{equation*}
As a consequence, we infer that
\begin{equation*}
\begin{split}
\big|\|u(\cdot+z)&-u\|_{L^p(w^0)}-\|z\cdot Du\|_{L^p(w^0)}\big|^p
\le
\|u(\cdot+z)-u-z\cdot Du\|_{L^p(w^0)}^p
\\
&=
\frac1{2^p}\int_{\R^N}\bigg|\int_0^1\left(D^2u(x+tz)[z]\cdot
z\right)\,2(1-t)\di t\,\bigg|^p\,w^0(x)\di x
\\
&\le
\frac{|z|^{2p}}{2^{p}}
\int_0^1\left(\int_{\R^N}|D^2u(x+tz)|^p\di x\right)\,2(1-t)\di t
=  
\frac{|z|^{2p}}{2^p}\,\|D^2u\|_{L^p}^p.
\end{split}
\end{equation*}
by Jensen's inequality.
The conclusion hence follows by combining the simple inequality
\begin{equation*}
\big|\|u(\cdot+z)-u\|_{L^p(w^z)}-\|u(\cdot+z)-u\|_{L^p(w^0)}\big|
\le 
\|u(\cdot+z)-u\|_{L^p}\|w^z-w^0\|_{L^\infty}^{1/p}
\end{equation*}
with~\eqref{eq:omega_w} and~\cite{GS25}*{Lem.~2.2(i)} (or, equivalently,
statement~\ref{item:ftc1} for $w=1$).
\end{proof}

We can now prove the pointwise $\limsup$ inequality in
\cref{res:bbm_w}\ref{item:bbm_w_limsup} in the case $w_k=w$ for every $k\in\N$
and $\|w\|_{L^\infty}=1$.

\begin{proof}[Proof of \cref{res:bbm_w}\ref{item:bbm_w_limsup}]
It is enough to show how the proof of~\cite{GS25}*{Th.~3.2(ii)} adapts to the
present setting.
From now on, the notation is the same one used in~\cite{GS25}, with the obvious
modifications.
In particular, we let
\begin{equation}
\label{eq:dado}
\mathscr F_{k,p}^{w_k}(u;A)
=
\int_A\frac{\|u(\cdot+z)-u\|_{L^p(w^z_k)}^p}{|z|^p}\,\rho_k(z)\di z
\end{equation}
for every $k\in\N$ and every measurable set $A\subset\R^N$.
By \cref{res:ftc}\ref{item:ftc1}, we can estimate
\begin{equation*}
\mathscr F_{k,p}^w(u;B_{\delta_l})
\le 
\int_{\mathbb S^{N-1}}
\|\sigma\cdot Du\|_{L^p(w^0)}^p\di\mu_k^l(\sigma)
+
\omega(\delta_l)\,
\|Du\|_{L^p}^p
\end{equation*}
for every $k,l\in\N$ (recall the notation in~\cite{GS25}*{Lem.~2.9}).
Since $\sigma\mapsto\|\sigma\cdot Du\|_{L^p(w^0)}^p\in C(\mathbb S^{N-1})$ and
$\delta_l\to0^+$ as $l\to\infty$, we then get
\begin{equation*}
\lim_{l\to\infty}
\limsup_{k\to\infty}
\mathscr F_{k,p}^w(u;B_{\delta_l})
\le 
\lim_{l\to\infty}
\lim_{k\to\infty}
\int_{\mathbb S^{N-1}}
\|\sigma\cdot Du\|_{L^p(w^0)}^p\di\mu_k^l(\sigma)
=
\mathscr D^\mu_{p,w}(u).
\end{equation*}
Moreover, since 
\begin{equation}
\label{eq:gufetto}
z\mapsto f_{u,w}(z)=\frac{\|u(\,\cdot+z)-u\|_{L^p(w^z)}^p}{|z|^p}\in
C(\R^N\setminus\set{0}),
\end{equation}
we get that (observe that now $\nu=\alpha\delta_0$ for some $\alpha\ge0$
by~\eqref{eq:ddp_alpha}) 
\begin{equation*}
\lim_{l\to\infty}
\limsup_{k\to\infty}
\mathscr F_{k,p}^w(u;A_l)
\le
\lim_{l\to\infty}
\lim_{k\to\infty}
\int_{A_l}
f_{u,w}\di\nu_k
=
\alpha\int_{\R^N\setminus\set*{0}}f_{u,w}\di\delta_0
=0.
\end{equation*}
Finally, we can estimate
\begin{equation*}
\lim_{l\to\infty}
\limsup_{k\to\infty}
\mathscr F_{k,p}^w(u;B_{1/\delta_l}^c)
\le
2^p\,\|u\|_{L^p}^p
\lim_{l\to\infty}
\limsup_{k\to\infty}
\int_{B_{1/\delta_l}^c}
\frac{\rho_{k}(z)}{|z|^p}\di z=0,
\end{equation*}
and the conclusion readily follows.
\end{proof}

We now pass to proof of the $\Gamma$-$\liminf$ inequality in
\cref{res:bbm_w}\ref{item:bbm_w_G-liminf} in the case $w_k=w$ for every
$k\in\N$ and $\|w\|_{L^\infty}=1$.
We let 
\begin{equation*}
w[\eta_j](x,y)=\int_{\R^N}w(x+z,y+z)\,\eta_j(z)\di z
\end{equation*}
for all $j\in\N$ and $x,y\in\R^N$.
It is not difficult to see that, by~\eqref{eq:omega_w}, 
\begin{equation}
\label{eq:w_moll}
\|w[\eta_j]-w\|_{L^\infty}
\le 
|B_1|\,\|\eta\|_{L^\infty}\,\omega\left(\tfrac2j\right)
\quad
\text{for all}\
j\in\N.
\end{equation}

In the proof of \cref{res:bbm_w}\ref{item:bbm_w_G-liminf}, we also need the
following lower semicontinuity result, which immediately follows
from~\eqref{eq:bruco} owing to Fatou's Lemma.

\begin{lemma}
\label{res:lsc_D_w}
If $u\in\sob{p}(\R^N)$ and $u^j=u*\eta_j$ for all $j\in\N$, then
\begin{equation*}
\liminf_{j\to\infty}
\mathscr D^\mu_{p,w}(u^j)
\ge 
\mathscr D^\mu_{p,w}(u).
\end{equation*}
\end{lemma}

\begin{proof}[Proof of \cref{res:bbm_w}\ref{item:bbm_w_G-liminf}]
It is enough to show how the proof of~\cite{GS25}*{Th.~3.2(iii)} adapts to the
present setting.
From now on, the notation is the same one used in~\cite{GS25}, with the obvious
modifications (also recall the notation~\eqref{eq:dado}).
We set $u_k^j=u_k*\eta_j$ and $u^j=u*\eta_j$ for every $k,j\in\N$.
We replace the use of Minkowski’s inequality at the beginning of the proof
of~\cite{GS25}*{Th.~3.2(iii)} with 
\begin{equation*}
\mathscr F_{k,p}^w(u_k)
\ge 
\mathscr F_{k,p}^w(u_k^j)
-
|B_1|\,\|\eta\|_{L^\infty}\,\omega\left(\tfrac2j\right)\mathscr F_{k,p}^1(u_k)
\end{equation*}
for all $k,j\in\N$, which readily follows by combining~\eqref{eq:w_diff_F}
and~\eqref{eq:w_moll}.  
Hence 
\begin{equation}
\label{eq:gintonic}
\liminf_{k\to\infty}
\mathscr F_{k,p}^w(u_k)
\ge
\liminf_{k\to\infty} 
\mathscr F_{k,p}^w(u_k^j)
-
C_{\eta}\,\omega\left(\tfrac2j\right)
\limsup_{k\to\infty}
\mathscr F_{k,p}^1(u_k)
\end{equation}
for all $j\in\N$. 
As in~\cite{GS25}, we now just have to show that 
\begin{equation}
\label{eq:poncio}
\lim_{k\to\infty}
\mathscr F_{k,p}^w(u_k^j)
=
\mathscr D^{\mu}_{p,w}(u^j)
\end{equation}
for every $j\in\N$, from which the conclusion follows by letting $j\to\infty$
in~\eqref{eq:gintonic}, thanks to \cref{res:lsc_D_w}.
The proof of~\eqref{eq:poncio} goes as in~\cite{GS25} so we only sketch it.
We first notice that 
\begin{equation*}
\lim_{l\to\infty}
\lim_{k\to\infty}
\mathscr F_{k,p}^w(u_k^j;B_{1/\delta_l}^c)
\le
2^p\,M\,\|u^j\|_{L^p}^p
\lim_{l\to\infty}
\delta^p_l=0
\end{equation*}
for every $j\in\N$, where $M>0$ depends on~\eqref{eq:ddp_p} only.
Moreover, with the same notation in~\eqref{eq:gufetto}, since $f_{u_k^j,w}\to
f_{u^j,w}$ locally uniformly on $\R^N\setminus\set*{0}$ as $k\to\infty$, we
also have that (recall that now $\nu=\alpha\delta_0$ for some $\alpha\ge0$
by~\eqref{eq:ddp_alpha})
\begin{equation*}
\begin{split}
\lim_{l\to\infty}
\lim_{k\to\infty}
\mathscr F_{k,p}^w(u_k^j;A_l)
&=
\lim_{l\to\infty}
\lim_{k\to\infty}
\int_{A_l}f_{u_k^j}\di\nu_k
=
\alpha\int_{\R^N\setminus\set*{0}}f_{u^j}\di\delta_0
=0
\end{split}
\end{equation*}
for every $j\in\N$.
Besides, since $\sigma\mapsto\|\sigma\cdot Du_k^j\|_{L^p(w^0)}^p$ converges
uniformly on $\mathbb S^{N-1}$ to $\sigma\mapsto\|\sigma\cdot
Du^j\|_{L^p(w^0)}^p$ as $k\to\infty$ for every $j\in\N$, we can infer that 
\begin{equation*}
\begin{split}
\lim_{l\to\infty}
\lim_{k\to\infty}
\int_{B_{\delta_l}}
&\left\|\tfrac{z}{|z|}\cdot Du_k^j\right\|_{L^p(w^0)}^p\di\nu_k(z)
=
\lim_{l\to\infty}
\lim_{k\to\infty}
\int_{\mathbb S^{N-1}}
\|\sigma\cdot Du_k^j\|_{L^p(w^0)}^p\di\mu_k^l(\sigma)
\\
&=
\lim_{l\to\infty}
\int_{\mathbb S^{N-1}}
\|\sigma\cdot Du^j\|_{L^p(w^0)}^p\di\mu^l(\sigma)
=
\int_{\mathbb S^{N-1}}
\|\sigma\cdot Du^j\|_{L^p(w^0)}^p\di\mu(\sigma)
\end{split}
\end{equation*}
for every $j\in\N$ (recall the notation in~\cite{GS25}*{Lem.~2.9}).
Hence the claim in~\eqref{eq:poncio} reduces to
\begin{equation*}
\lim_{l\to\infty}
\lim_{k\to\infty}
\mathscr F_{k,p}^w(u_k^j;B_{\delta_l})
=
\int_{\mathbb S^{N-1}}
\|\sigma\cdot Du^j\|_{L^p(w^0)}^p\di\mu(\sigma)
\end{equation*} 
or, equivalently,
\begin{equation}
\label{eq:maciste}
\lim_{l\to\infty}
\lim_{k\to\infty}
\left|
\mathscr F_{k,p}^w(u_k^j;B_{\delta_l})
-
\int_{B_{\delta_l}}\left\|\tfrac{z}{|z|}\cdot
Du_k^j\right\|_{L^p(w^0)}^p\di\nu_k(z)
\right|
=
0
\end{equation}
for every $j\in\N$.
To show~\eqref{eq:maciste}, we observe that, by \cref{res:ftc}\ref{item:ftc2},
we can estimate  
\begin{equation*}
\left|
\frac{\|u_k^j(\,\cdot+z)-u_k^j\|_{L^p(w^z)}}{|z|}
-
\left\|\tfrac{z}{|z|}\cdot Du_k^j\right\|_{L^p(w^0)}
\right|
\le 
\frac{|z|}{2}\,\|D^2u_k^j\|_{L^p}
+
\omega(|z|)^{1/p}\,\|Du_k^j\|_{L^p}\end{equation*}
for every $z\in\R^N$.
Hence, since $|a^p-b^p|\le p\max\set*{a,b}^{p-1}|a-b|$, for all $a,b\ge0$, and 
\begin{equation*}
\max\set*{
\frac{\|u_k^j(\,\cdot+z)-u_k^j\|_{L^p(w^z)}}{|z|},
\left\|\tfrac{z}{|z|}\cdot Du_k^j\right\|_{L^p(w^0)}
}
\le
\|Du_k^j\|_{L^p}
\end{equation*}
for every $k,j\in\N$ and $z\in\R^N$ by \cref{res:ftc}\ref{item:ftc1}, we can
estimate
\begin{equation*}
\begin{split}
\bigg|
&\frac{\|u_k^j(\,\cdot+z)-u_k^j\|_{L^p(w^z)}^p}{|z|^p}
-
\left\|\tfrac{z}{|z|}\cdot Du_k^j\right\|_{L^p(w^0)}^p
\bigg|
\\
&\qquad\le
p\,
\|Du_k^j\|_{L^p}^{p-1}
\,
\left|
\frac{\|u_k^j(\,\cdot+z)-u_k^j\|_{L^p(w^z)}}{|z|}
-
\left\|\tfrac{z}{|z|}\cdot Du_k^j\right\|_{L^p(w^0)}
\right|
\\
&\qquad\le 
\frac{p|z|}{2}
\,
\|Du_k^j\|_{L^p}^{p-1}
\,
\|D^2u_k^j\|_{L^p}
+
p\,
\omega(|z|)^{1/p}\,\|Du_k^j\|_{L^p}^p,
\end{split}
\end{equation*} 
for every $i,j\in\N$ and $z\in\R^N$.
We thus obtain that
\begin{equation*}
\begin{split}
\bigg|
\mathscr F_{t_k,p}(u_k^j;B_{\delta_l})
-
\int_{B_{\delta_l}}\left\|\tfrac{z}{|z|}\cdot Du_k^j\right\|_{L^p}^p\di\nu_k(z)
\bigg|
&\le 
\frac{p}{2}
\,
\|Du_k^j\|_{L^p}^{p-1}
\,
\|D^2u_k^j\|_{L^p}
\,
\delta_l\,\nu_k(B_{\delta_l})
\\
&\quad+
p\,\omega(\delta_l)^{1/p}\|Du_k^j\|_{L^p}^p\,\nu_k(B_{\delta_l})
\end{split}
\end{equation*}
for every $k,j,l\in\N$, from which the claimed~\eqref{eq:maciste} readily
follows, concluding the proof.
\end{proof}

We finally prove  \cref{res:bbm_w}\ref{item:bbm_w_G-liminf_R} by reducing to
\cref{res:bbm_w}\ref{item:bbm_w_G-liminf} by means of the following simple
estimate.

\begin{lemma}
\label{res:ell}
Assume that $w>0$ on $\R^{2n}$, let $R>0$ and set 
\begin{equation}
\label{eq:ell_R}
\ell_R
=
\inf\set*{w(x,x+z):x\in B_R,\ z\in B_{2R}}>0.
\end{equation}
If $u\in L^p_R(\R^N)$ and $k\in\N$ is such that
$\|w_k-w\|_{L^\infty}\le\frac{\ell_R}2$, then 
\begin{equation*}
\mathscr F^1_{k,p}(u)
\le
\frac2{\ell_R}
\,
\mathscr F^{w_k}_{k,p}(u)
+
2^p\|u\|_{L^p}^p
\int_{B_{2R}^c}\frac{\rho_k(z)}{|z|^p}\di z.
\end{equation*}
\end{lemma}

\begin{proof}
Note that, if $z\in B_{2R}$ and $x\in B_R^c$, then also $x+z\in B_R^c$. 
Therefore, we have
\begin{equation*}
\begin{split}
\mathscr F^1_{k,p}(u)
&\le
\int_{B_{2R}}\frac{\|u(\cdot+z)-u\|_{L^p}^p}{|z|^p}\,\rho_k(z)\di z
+
2^p\|u\|_{L^p}^p
\int_{B_{2R}^c}\frac{\rho_k(z)}{|z|^p}\di z
\\
&= 
\int_{B_{2R}}\frac{\|u(\cdot+z)-u\|_{L^p(B_R)}^p}{|z|^p}\,\rho_k(z)\di z
+
2^p\|u\|_{L^p}^p
\int_{B_{2R}^c}\frac{\rho_k(z)}{|z|^p}\di z.
\end{split}
\end{equation*}
The conclusion hence follows in virtue of the definition in~\eqref{eq:ell_R},
since the assumption that $\|w_k-w\|_{L^\infty}\le\frac{\ell_R}2$ implies that

\begin{equation*}
w_k(x,x+z)
\ge 
w(x,x+z)-|w_k(x,x+z)-w(x,x+z)|
\ge 
\ell_R
-
\frac{\ell_R}2
=
\frac{\ell_R}2
\end{equation*}
for every $x\in B_R$ and $z\in B_{2R}$.
\end{proof}

\begin{proof}[Proof of \cref{res:bbm_w}\ref{item:bbm_w_G-liminf_R}]
Without loss of generality and possibly passing to a non-re\-la\-bel\-led
subsequence, we can assume that 
\begin{equation*}
\liminf_{k\to\infty}\mathscr F^{w_k}_{k,p}(u_k)
=
\lim_{k\to\infty}\mathscr F^{w_k}_{k,p}(u_k)
<\infty.
\end{equation*}
By applying \cref{res:ell} to each $u_k\in L^p_R(\R^N)$ for every $k\in\N$
sufficiently large depending on $R>0$, we get that (where $\ell_R>0$ is defined as in~\eqref{eq:ell_R} in \cref{res:ell})
\begin{equation*}
\limsup_{k\to\infty}
\mathscr F^1_{k,p}(u_k)
\le
\frac2{\ell_R}
\,
\lim_{k\to\infty}\mathscr F^{w_k}_{k,p}(u_k)
+
\frac{M}{R^p}\,\|u\|_{L^p}^p<\infty,
\end{equation*} 
where $M>0$ depends on~\eqref{eq:ddp_p} only.
The conclusion hence follows by \cref{res:bbm_w}\ref{item:bbm_w_G-liminf}.
\end{proof}

\subsection{Proof of \texorpdfstring{\cref{res:comp_w}}{Theorem 1.5}}

The proof of \cref{res:comp_w} is a simple combination of \cref{def:p-compact}
and \cref{res:ell}.

\begin{proof}[Proof of \cref{res:comp_w}]
Let $R>0$ be fixed and let $\ell_R>0$ be given by~\eqref{eq:ell_R}.
Let $k_R\in\N$ be such that $\|w_k-w\|_{L^\infty}\le\frac{\ell_R}2$ for every
$k\ge k_R$. 
By \cref{res:ell}, we can thus estimate
\begin{equation*}
\sup_{k\ge k_R}
\mathscr F^1_{k,p}(u_k)
\le 
\frac{2}{\ell_R}
\sup_{k\ge k_R}
\mathscr F^{w_k}_{k,p}(u_k)
+
M
\,\sup_{k\in\N}\|u_k\|_{L^p}<\infty,
\end{equation*}
where $M>0$ depends on~\eqref{eq:ddp_p} only.
Therefore, we get that 
\begin{equation*}
\sup_{k\in\N}
\left(
\|u_k\|_{L^p}
+
\mathscr F^1_{k,p}(u_k)
\right)
<\infty
\end{equation*}
and thus, by \cref{def:p-compact}, we infer that $(u_k)_{k\in\N}$ is compact in $L^p(E)$ for every compact set
$E\subset\R^N$ and any of its $L^p_{\rm loc}(\R^N)$ limits is in
$\sob{p}(\R^N)$.
The conclusion hence immediately follows by recalling that
$(u_k)_{k\in\N}\subset \sob{p}_R(\R^N)$.
\end{proof}

\subsection{Proof of \texorpdfstring{\cref{res:poincare_w}}{Theorem 1.9}}

We establish \cref{res:poincare_w} by exploiting \cref{res:strong_bbm_w}.

\begin{proof}[Proof of \cref{res:poincare_w}]
By contradiction,  we can assume that there exists $\e>0$ and
$(u_k)_{k\in\N}\subset L^p_0(\Omega)$ such that $\|u_k\|_{L^p}^p=1$ for all
$k\in\N$ and  
\begin{equation*}
\mathscr F_{k,p}^{w_k}(u_k)
<
\frac{1}{A_{p,\mu,\Omega,w}+\e}
\end{equation*} 
for all $k\in\N$.
By~\ref{item:strong_bbm_comp} and~\ref{item:strong_bbm_liminf} in 
\cref{res:strong_bbm_w}, up to passing to a subsequence, there exists $u\in\sob
p_0(\Omega)$ such that $u_k\to u$ in $L^p(\R^N)$ as $k\to\infty$ and 
\begin{equation*}
\mathscr D^\mu_{p,w}(u)
\le 
\liminf_{k\to\infty}\mathscr F^{w_k}_{k,p}(u_k)
\le 
\frac{1}{A_{p,\mu,\Omega,w}+\e}.
\end{equation*}
This, in turn, in combination with~\eqref{eq:poincare_mu_w} with the optimal
constant $C=A_{p,\mu,\Omega,w}$, implies that
$A_{p,\mu,\Omega,w}+\e<A_{p,\mu,\Omega,w}$, which is a contradiction, since
$\e>0$ by assumption.
\end{proof}

\subsection{Proof of \texorpdfstring{\cref{res:eigen_conv}}{Theorem 1.10}}

The proof of \cref{res:eigen_conv} is a plain application of \cref{res:strong_bbm_w} and exploits a simple argument which is at heart of many results of~\cites{SS24,FPSS24}.

\begin{proof}[Proof of \cref{res:eigen_conv}]
We split the proof into two parts, proving the convergence of the eigenvalues in~\eqref{eq:eigenv_conv} and of the eigenfunctions in~\eqref{eq:eigenf_conv} separately.

\vspace{1ex}

\textit{Part 1: proof of~\eqref{eq:eigenv_conv}}.
On the one hand, given $\e>0$, we can find $u\in\sob p_0(\Omega)$ such that
$\|u\|_{L^p}=1$ and 
\begin{equation*}
\mathscr D_{p,w}^\mu(u) 
\le 
\lambda_{p,w}^\mu(\Omega)
+
\e.
\end{equation*}
Since $\lambda_{k,p}^{w_k}(\Omega)\le\mathscr F^{w_k}_{k,p}(u)$ for all
$k\in\N$, by \cref{res:strong_bbm_w}\ref{item:strong_bbm_limsup} we thus infer
that
\begin{equation*}
\limsup_{k\to\infty}
\lambda_{k,p}^{w_k}(\Omega)
\le 
\limsup_{k\to\infty}
\mathscr F^{w_k}_{k,p}(u)
\le 
\mathscr D_{p,w}^\mu(u) 
\le 
\lambda_{p,w}^\mu(\Omega)
+
\e
\end{equation*}
from which we immediately get that 
\begin{equation}
\label{eq:sole}
\limsup_{k\to\infty}
\lambda_{k,p}^{w_k}(\Omega)
\le 
\lambda_{p,w}^\mu(\Omega).
\end{equation}
In particular, we deduce that 
\begin{equation}
\label{eq:bradipo}
\sup_{k\in\N}
\lambda_{k,p}^{w_k}(\Omega)
<\infty.
\end{equation}
On the other hand, for each $k\in\N$, we can find $u_k\in L^p_0(\Omega)$ with
$\|u_k\|_{L^p}=1$ such that 
\begin{equation}
\label{eq:tenda}
\mathscr F^{w_k}_{k,p}(u_k)
\le 
\lambda_{k,p}^{w_k}(\Omega)
+
\frac1k.
\end{equation}
Owing to~\eqref{eq:bradipo}, we thus get that $(u_k)_{k\in\N}\subset
L^p_0(\Omega)$ satisfies~\eqref{eq:strong_bbm_comp}.
Therefore,
by~\ref{item:strong_bbm_comp} and~\ref{item:strong_bbm_liminf} in
\cref{res:strong_bbm_w}, there exist $(u_{k_j})_{j\in\N}$ and  $u\in\sob p(\Omega)$ such that $u_{k_j}\to u$ in $L^p(\R^N)$ as
$k\to\infty$, and thus $u\in\sob p_0(\Omega)$ with $\|u\|_{L^p}=1$, and, also owing to~\eqref{eq:tenda} and~\eqref{eq:sole},  
\begin{equation}
\label{eq:luna}
\mathscr D^\mu_{p,w}(u)
\le 
\liminf_{j\to\infty}
\mathscr F^{w_{k_j}}_{k_j,p}(u_{k_j})
\le 
\liminf_{k\to\infty}
\mathscr F^{w_k}_{k,p}(u_k)
\le
\liminf_{k\to\infty}
\lambda_{k,p}^{w_k}(\Omega)
\le 
\lambda^\mu_{p,w}(\Omega).
\end{equation}
Since $u\in\sob p_0(\Omega)$ is such that $\|u\|_{L^p}=1$, we must have that $\lambda^\mu_{p,w}(\Omega)\le\mathscr D^\mu_{p,w}(u)$ and thus all inequalities in~\eqref{eq:luna} are equalities, yielding 
\begin{equation}
\lambda^\mu_{p,w}(\Omega)
=
\liminf_{k\to\infty}
\lambda_{k,p}^{w_k}(\Omega)
\end{equation} 
which, combined with~\eqref{eq:sole}, yields~\eqref{eq:eigenv_conv}.

\vspace{1ex}

\textit{Part 2: proof of~\eqref{eq:eigenf_conv}}.
By assumption, we have $\|u_k\|_{L^p}=1$ and $\mathscr F^{w_k}_{k,p}(u_k)=\lambda^{w_k}_{k,p}(\Omega)$ for all $k\in\N$. 
Owing to~\eqref{eq:eigenv_conv}, we thus infer the validity of~\eqref{eq:strong_bbm_comp}, and so, by \cref{res:strong_bbm_w}\ref{item:strong_bbm_comp}, there exist $(u_{k_j})_{j\in\N}$ and $u\in\sob p_0(\Omega)$ satisfying~\eqref{eq:eigenf_conv}.
We are thus left to prove that $u\in\sob p_0(\Omega)$ is an eigenfunction relative to $\lambda^\mu_{p,w}(\Omega)$.
Indeed, by~\eqref{eq:eigenf_conv}, we know that $\|u\|_{L^p}=1$. 
Moreover, by \cref{res:strong_bbm_w}\ref{item:strong_bbm_liminf}, we also infer that  
\begin{equation}
\mathscr D^\mu_{p,w}(u)
\le 
\liminf_{j\to\infty}
\mathscr F^{w_{k_j}}_{k_j,p}(u_{k_j})
\le 
\liminf_{k\to\infty}
\mathscr F^{w_k}_{k,p}(u_k)
=
\liminf_{k\to\infty}
\lambda^{w_k}_{k,p}(\Omega)
=
\lambda^\mu_{p,w}(\Omega),
\end{equation}
from which we obtain $\mathscr D^\mu_{p,w}(u)=\lambda^\mu_{p,w}(\Omega)$, concluding the proof.
\end{proof}

\subsection{Proof of \texorpdfstring{\cref{res:J}}{Theorem 1.2}}

We now pass to the proof of \cref{res:J}.
Since the argument follows the one of the proof of~\cite{GS25}*{Th.~5.1} with minor
changes, we only sketch it for the ease of the reader.

\begin{proof}[Proof of \cref{res:J}]
We prove the three statements separately.

\vspace{1ex}

\textit{Proof of~\ref{item:J_comp}}.
Let $E\subset A$ be a compact set and let $\e>0$.
Let $\delta>0$ be given by~\eqref{eq:J_comp} and also such that $E_\delta\subset A$, where 
\begin{equation*}
E_\delta=\bigcup_{x\in E}B_\delta(x).
\end{equation*} 
Letting $\eta_\delta=\chi_{B_\delta}/|B_\delta|$ and arguing as in the proof of~\cite{GS25}*{Lem.~5.3}, we can estimate
\begin{equation}
\label{eq:brez}
\begin{split}
\|\eta_\delta*v-v\|_{L^p(E)}^p
&\le 
\frac{1}{|B_\delta|}
\int_E\int_{B_\delta}
|v(x+z)-v(x)|^p
\di z\di x
\\
&= 
\frac{1}{|B_1|}
\int_E\int_{B_\delta(x)}
\frac{|v(y)-v(x)|^p}{\delta^N}
\di y\di x
\\
&\le 
\frac{\e}{|B_1|}
\int_E\int_{B_\delta(x)}
|v(y)-v(x)|^p\,J_k(x,y)
\di y\di x
\le
\e
\mathscr J_{k,p}(v) 
\end{split}
\end{equation}
for every $v\in L^p(\R^N)$ and $k\in\N$, owing to~\eqref{eq:J_comp}.
Now, if $(u_k)_{k\in\N}\subset L^p(A)$ satisfies~\eqref{eq:J_sup_comp}, then the sequence $(v_k^E)_{k\in\N}\subset L^p(\R^N)$, where, for all $k\in\N$, \begin{equation*}
v_k^E
=
\begin{cases}
u_k & \text{on}\ E,
\\[1ex]
0 & \text{otherwise},
\end{cases}
\end{equation*} 
is bounded in $L^p(\R^N)$. 
Therefore, by~\cite{Brezis11}*{Cor.~4.28}, for every $\delta>0$ sufficiently small, $(\eta_\delta*v_k^E)_{k\in\N}$ is compact in $L^p(E)$, and thus totally bounded in $L^p(E)$.
By~\eqref{eq:brez}, also   $(u_k)_{k\in\N}\subset L^p(E)$ is thus totally bounded  in $L^p(E)$, and therefore compact in $L^p(E)$.

\vspace{1ex}

\textit{Proof of \ref{item:J_G-liminf}}.
If $(u_k)_{k\in\N}\subset L^p(A)$ is such that $u_k\to
u$ in $L^p(A)$ as $k\to\infty$, then, up to a subsequence, by~\eqref{eq:J_liminf} we get that 
\begin{equation*}
\liminf_{k\to\infty}
|u_k(x)-u_k(y)|^p\,J_k(x,y)
\ge
|u(x)-u(y)|^p\,J(x,y)
\end{equation*}
for $\mathscr L^{2n}$-a.e.\ $(x,y)\in A^2$, so that the conclusion follows
by Fatou's Lemma.

\vspace{1ex}

\textit{Proof of \ref{item:J_lim}}.
If $u\in L^p(A)$ is such that $\mathscr J_p(u)<\infty$, then 
\begin{equation*}
(x,y)\mapsto|u(x)-u(y)|^pJ_p(x,y)\in L^1(A^2).
\end{equation*}
Therefore, by~\eqref{eq:J_lim}, the conclusion follows either by the Dominated
Convergence Theorem under~\eqref{eq:J_dom}, or by the Monotone Convergence
Theorem under~\eqref{eq:J_mon}.
\end{proof}

\subsection{Proof of \texorpdfstring{\cref{res:nonloc_bbm_w}}{Corollary 1.10}}

We conclude with a brief sketch the proof of \cref{res:nonloc_bbm_w} for the ease of the reader.  

\begin{proof}[Proof of \cref{res:nonloc_bbm_w}]
To prove statements~\ref{item:nonloc_G-liminf} and~\ref{item:nonloc_lim} of \cref{res:nonloc_bbm_w}, we can directly apply the corresponding statements~\ref{item:J_G-liminf} and~\ref{item:J_lim} of~\cref{res:J} with the choices
\begin{equation*}
A=\R^N,
\quad
J_k(x,y)=\frac{\rho_k(x-y)}{|x-y|^p}\,w_k(x,y)
\quad
\text{and}
\quad
J(x,y)=\kappa(x-y)\,w(x,y)
\end{equation*}
for all $x,y\in\R^N$ and $k\in\N$, since clearly~\eqref{eq:nonloc_bbm_w_ass} implies the validity of~\eqref{eq:J_lim} and~\eqref{eq:J_dom}.
To prove  statement~\ref{item:J_comp} of \cref{res:nonloc_bbm_w}, instead, it is enough to observe that
\begin{equation*}
\inf\set*{w_k(x,y) : x\in B_R,\ y\in B_1(x)}
\ge 
\frac12\inf\set*{w(x,y) : x\in B_R,\ y\in B_1(x)}>0
\end{equation*} 
for all $k\in\N$ sufficiently large
whenever $R>0$ is fixed, and then argue as in the proof of statement~\ref{item:J_comp} of \cref{res:J}.
We omit the analogous computations.
\end{proof}


\end{document}